\documentclass[12pt,epsfig,amsfonts]{amsart}
\usepackage{amsmath,amsthm,amssymb,amscd,epsfig,color,bbm}
\usepackage{graphicx}
\usepackage{url}
\usepackage{mathrsfs}
\usepackage{ulem}
\usepackage{pb-diagram} 
\usepackage[all]{xy}

\newtheorem{prop}{Proposition}[section]
\newtheorem{lemma}[prop]{Lemma}

\newtheorem{thm}[prop]{Theorem}

\theoremstyle{definition}

\theoremstyle{remark}

\numberwithin{equation}{section}

\begin{document}

\author{Hiroki Takahasi}

 \address{Department of Mathematics,
 Keio University, Yokohama,
 223-8522, JAPAN} 
 \email{hiroki@math.keio.ac.jp}


\date{}

\title[Distributions of periodic points]
 {Distributions of periodic points for the Dyck shift and the heterochaos baker maps}

 \maketitle

 \begin{abstract}
The heterochaos baker maps are piecewise affine maps on the square or the cube that  
are one of the simplest partially hyperbolic systems. 
 The Dyck shift is a well-known example of a subshift that has two fully supported ergodic measures of maximal entropy (MMEs). 
We show that the two ergodic MMEs of the Dyck shift  
are represented as asymptotic distributions of sets of periodic points of different multipliers. 
We transfer this result to the heterochaos baker maps, and show that their two ergodic MMEs are represented as asymptotic distributions of sets of periodic points of different unstable dimensions.

      \end{abstract}

\section{Introduction}
Let $X$ be a topological space and let
 $T\colon X\to X$ be a Borel map. 
 For $n\in\mathbb N$,
  elements of the set
 ${\rm Per}_n(T)=\{x\in X\colon T^nx=x\}$ are called {\it periodic points of period $n$} of $T$.  
 When $X$ is a differentiable manifold, 
 we say $x\in {\rm Per}_n(T)$ is {\it hyperbolic} if $T^n$ is differentiable on a neighborhood of $x$ and all the eigenvalues of the derivative $DT^n(x)$ lie outside of the unit circle. 
 Infinitely many hyperbolic periodic orbits are embedded in chaotic dynamical systems, and they can be used as a spine to structure the dynamics.
 A hyperbolic periodic point $x\in {\rm Per}_n(T)$ is said to be {\it $k$-unstable} $(1\leq k\leq \dim X)$ if the number of the eigenvalues of $DT^n(x)$ counted with multiplicity that lie outside of the unit circle is $k$.
 If $x\in {\rm Per}_n(T)$ is $k$-unstable, $k$ is called the {\it unstable dimension} of $x$.

 Let $M(X)$ denote the space of Borel probability measures on $X$ endowed with the weak* topology and 
 let $M(X,T)$ denote the subspace of $M(X)$ that consists of $T$-invariant elements. For each $\mu\in M(X,T)$, let $h(\mu,T)\in[0,\infty]$ 
 denote the measure-theoretic  
entropy of $\mu$ with respect to $T$. 
If $\sup\{h(\mu,T)\colon\mu\in M(X,T)\}$ is finite,
measures that attain this supremum are called {\it measures of maximal entropy} (MMEs).
In the thermodynamic formalism \cite{Rue04},
the non-uniqueness of MME is interpreted as phase transitions. One can advance one's knowledge on phase transitions by analyzing phenomena associated with the non-uniqueness of MME.

For each $n\in\mathbb N$ with ${\rm Per}_n(T)\neq\emptyset$, consider the probability measure
 \[\mu_{T,n}=\frac{\sum_{x\in {\rm Per}_{n}(T)}\delta_x}{\#{\rm Per}_{n}(T) },\]
 where $\delta_x\in M(X)$ denotes the unit point mass at $x\in X$.  
 If $T$ is a transitive uniformly hyperbolic (Axiom A) diffeomorphism,
  the unstable dimension of periodic points of $T$ is constant, and  $\{\mu_{T,n}\}$ converges to the unique MME \cite{Bow71}. 
  In this paper we establish an analogue of this convergence for some simple 
  partially hyperbolic systems 
  for which MMEs are not unique.
  
  For partially hyperbolic diffeomorphisms, periodic points with different unstable dimensions can coexist densely in the same transitive set \cite{ABCDW,AS70,BD96,Man78,Shu71,Si72}.
Further, MMEs need not be unique \cite{BFT23,NMRV,RT22,RRTU12}. 
Then a natural question is which periodic points are to be used to represent each of the coexisting MMEs.
We would like to shed some light on this naive question by analyzing simple systems,
 called the {\it heterochaos baker maps}, introduced in \cite{STY21} and later in \cite{TY23} in a slightly more general form.  They are piecewise affine maps on the square or the cube, not a diffeomorphism, but retain some features of general partially hyperbolic diffeomorphisms. Below we introduce these maps, and state a main result.
\subsection{Distributions of periodic points for the heterochaos baker maps}\label{hetero-b-sec} 

Let $M\geq2$ be an integer.
To define the hetrochaos baker maps $f_a\colon[0,1]^2\to[0,1]^2$ and $f_{a,b}\colon[0,1]^3\to[0,1]^3$, where parameters $a$, $b$ range over the interval $(0,\frac{1}{M})$,
we write $(x_u,x_c)$ and $(x_u, x_c, x_s)$ for the coordinates on $[0,1]^2$ and $[0,1]^3$ respectively.
Define $\tau_a\colon[0,1]\to[0,1]$ by 
\[\tau_a(x_u)=\begin{cases}\vspace{1mm}
\displaystyle{\frac{x_u-(k-1)a}{a}}&\text{ on }[(k-1)a,ka),\ k=1,\ldots,M,\\ \displaystyle{\frac{x_u-Ma}{1-Ma}}&\text{ on }
[Ma,1].
\end{cases}\]
We introduce two alphabets consisting of $M$ symbols
\[D_\alpha=\{\alpha_1,\ldots,\alpha_M\}\ \text{ and }\ D_\beta=\{\beta_1,\ldots,\beta_M\},\] and set
 $D=D_\alpha\cup D_\beta.$ 
 For each $\gamma\in D$ we define a domain $\Omega_\gamma^+$ in $[0,1]^2$ by
\[\Omega_{\alpha_k}^+=\left[(k-1)a,ka\right)\times
\left[0,1\right]\ \text{ for }k=1,\ldots,M,\]
and
\[\Omega_{\beta_k}^+=\begin{cases}
\vspace{1mm}\displaystyle{\left[Ma,1\right]\times
\left[\frac{k-1}{M },\frac{k }{M }\right)}&
\text{ for }k=1,\ldots,M-1,\\
\displaystyle{\left[Ma,1\right]\times
\left[\frac{k-1 }{M},1\right]}&\text{ for }k=M.
\end{cases}\]
Define $f_{a}\colon [0,1]^2\to[0,1]^2$ by
\[\begin{split}
  f_a(x_u,x_c)=
  \begin{cases}
\displaystyle{\left(\tau_a(x_u),\frac{x_c}{M}+\frac{k-1}{M}\right)}&\text{ on }\Omega_{\alpha_k}^+,\ k=1,\ldots,M,\\
   \displaystyle{\left (\tau_a(x_u),Mx_c-k+1\right)}&\text{ on }\Omega_{\beta_k}^+,\ 
   k=1,\ldots,M.
   \end{cases}
\end{split}\]
\begin{figure}
\begin{center}
\includegraphics
[height=4cm,width=11cm]
{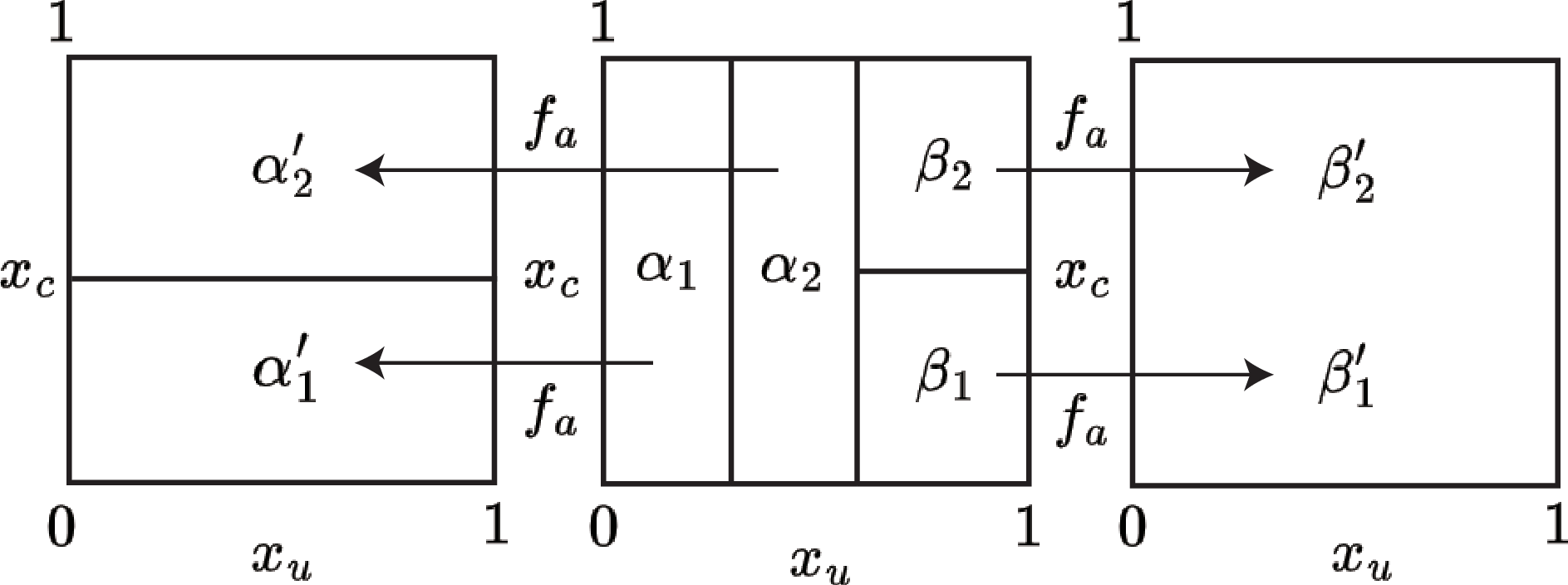}
\caption
{The map $f_{a}$ with $M=2$. For each $\gamma\in D$,
the domain $\Omega_\gamma^+$  and its image are labeled with $\gamma$ and $\gamma'$ respectively: $f_a(\Omega_{\beta_1}^+)=[0,1]\times[0,1)$ and $f_a(\Omega_{\beta_2}^+)=[0,1]^2$.}\label{fig1}
\end{center}
\end{figure}
\begin{figure}[b]
\begin{center}
\includegraphics
[height=4.4cm,width=12cm]
{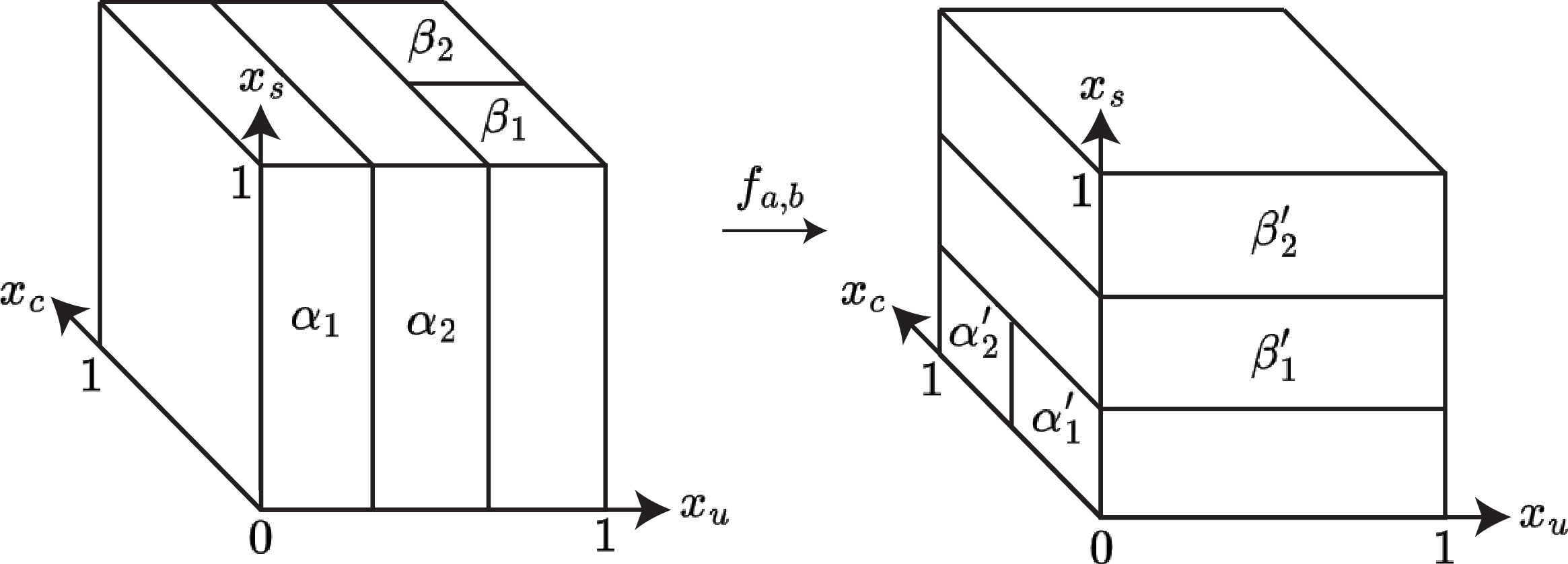}
\caption
{The map $f_{a,b}$ with $M=2$. For each $\gamma\in D$,
the domain $\Omega_\gamma$  and its image are labeled with $\gamma$ and $\gamma'$ respectively. }\label{fig2}
\end{center}
\end{figure}
Next, put $\Omega_{\alpha_k}=\Omega_{\alpha_k}^+\times\left[0,1\right]$ and $\Omega_{\beta_k}=\Omega_{\beta_k}^+\times\left[0,1\right]$
 for $k=1,\ldots,M$. Define $f_{a,b}\colon [0,1]^3\to[0,1]^3$ by 
\[\begin{split}
  f_{a,b}(x_u,x_c,x_s)=
  \begin{cases}
    \displaystyle{\left(f_{a}(x_u,x_c),
    (1-Mb)x_s\right)}&\text{ on }\Omega_{\alpha_k},\ k=1,\ldots,M,\\
   \displaystyle{\left (f_{a}(x_u,x_c),bx_s+1+b(k-M-1)\right)}&\text{ on }\Omega_{\beta_k},\ k=1,\ldots,M.
   \end{cases}
\end{split}\]
See \textsc{Figures}~\ref{fig1} and \ref{fig2} for the case $M=2$.  
Under the forward iteration of $f=f_{a,b}$,
the $x_u$-direction is expanding by factor $\frac{1}{a}$ or $\frac{1}{1-Ma}$ and the $x_s$-direction is contracting by factor $1-Mb$ or $b$. The $x_c$-direction is a center: contracting by factor $\frac{1}{M}$ on $\bigcup_{k=1}^M\Omega_{\alpha_k}$ and expanding by factor $M$ on $\bigcup_{k=1}^M\Omega_{\beta_k}$.
The map $f_a$ is the projection of $ f_{a,b}$ to the $(x_u,x_c)$-plane. 

Let ${\rm int}(\cdot)$ denote the interior operation in $\mathbb R^3$. For $f=f_{a,b}$ let $\Lambda=\Lambda_{a,b}$ denote the maximal $f$-invariant set given by
\begin{equation}\label{maximal-def}\Lambda=\bigcap_{n=-\infty}^\infty f^{-n}\left(\bigcup_{\gamma\in D}{\rm int}(\Omega_\gamma)\right).\end{equation}
We consider periodic points of $f|_\Lambda\colon\Lambda\to\Lambda$.
For each $n\in\mathbb N$ let ${\rm Per}_{\alpha,n}(f)$ (resp. ${\rm Per}_{\beta,n}(f)$) denote the set of $1$-unstable 
(resp. $2$-unstable) periodic points of $f|_\Lambda$ of period $n$, which are finite sets.
We exclude from further consideration periodic points of $f|_\Lambda$ that are not hyperbolic. The set of such periodic points contains continua parallel to the $x_c$-axis. 

Any heterochaos baker map $f\colon[0,1]^3\to[0,1]^3$ has the following properties:
see \cite[Theorem~1.1]{STY21} for (i); see \cite[Theorem~2.3]{STYY} and \cite{TY23}  for (ii).

 \begin{itemize}\item[(i)]Both $\bigcup_{n\in\mathbb N}{\rm Per}_{\alpha,n}(f)$ and $\bigcup_{n\in\mathbb N}{\rm Per}_{\beta,n}(f)$ are dense in $[0,1]^3$. This is the reason why $f$ is called `heterochaos'.

 \item[(ii)] 
There exist exactly two
 ergodic MMEs of entropy $\log (M+1)$, denoted by $\mu_{\alpha}$ and $\mu_{\beta}$. They are  Bernoulli, 
 charge any non-empty open subset of $[0,1]^3$ and
 satisfy
\[\mu_{\alpha}(\Omega_{\alpha_k})=
\mu_{\beta}(\Omega_{\beta_k})=\frac{1}{M+1}\ \text{ for  }k=1,\ldots,M.\]
\end{itemize} 
In \cite{STY21,TY23}, (i) (ii) were proved under some restrictions on $(a,b)$. Actually these restrictions can be removed, see \cite{STYY}.

\begin{thm}\label{theorema}
Let $f\colon[0,1]^3\to[0,1]^3$ be a heterochaos baker map.
 For any continuous function $\varphi\colon[0,1]^3\to\mathbb R$ we have
\[\lim_{n\to\infty}\frac{\sum_{x\in{\rm Per}_{\alpha,n}(f)}\varphi(x)}{\#{\rm Per}_{\alpha,n}(f)}=\int\varphi{\rm d}\mu_{\alpha}\ \text{ and }  \ \lim_{n\to\infty}\frac{\sum_{x\in{\rm Per}_{\beta,n}(f)}\varphi(x)}{\#{\rm Per}_{\beta,n}(f)}=\int\varphi{\rm d}\mu_{\beta},\]
and
\[\lim_{n\to\infty}\frac{\sum_{x\in{\rm Per}_{\alpha,n}(f)\cup {\rm Per}_{\beta,n}(f)  }\varphi(x)}{\#({\rm Per}_{\alpha,n}(f)\cup {\rm Per}_{\beta,n}(f) )}=\frac{1}{2}\int\varphi{\rm d}\mu_{\alpha}+\frac{1}{2}\int\varphi{\rm d}\mu_{\beta}.\]
\end{thm}
Theorem~\ref{theorema} settles \cite[Conjecture~2.5]{STYY} in the affirmative.
Since the two ergodic MMEs of $f_{a,b}$ project to that of $f_a$, and there is a one-to-one correspondence between periodic points of $f_{a,b}$ in $\Lambda_{a,b}$ and that of $f_a$ in the projection of $\Lambda_{a,b}$, 
a statement analogous to Theorem~\ref{theorema} holds for $f_a$.

\begin{figure}[b]
 \begin{minipage}[b]
 {0.32\linewidth}
{\includegraphics[width=0.9\textwidth]{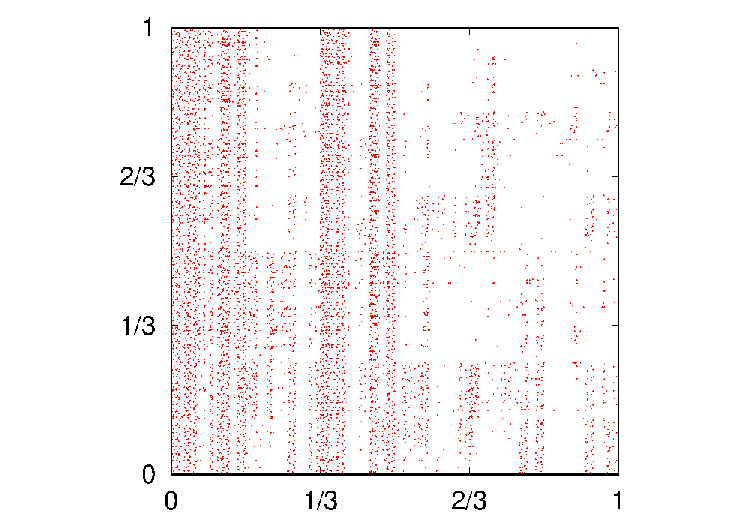}}
 {\small  period 11, 1-unstable}
\end{minipage}
\begin{minipage}[b]
{0.32\linewidth}
{\includegraphics[width=0.9\textwidth]{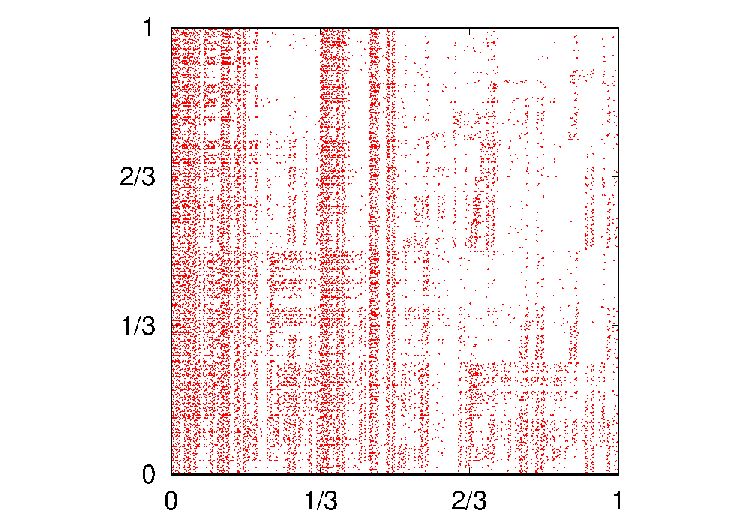}}
{\small period 12, 1-unstable}
\end{minipage}
\begin{minipage}[b]
{0.32\linewidth}
{\includegraphics[width=0.9\textwidth]{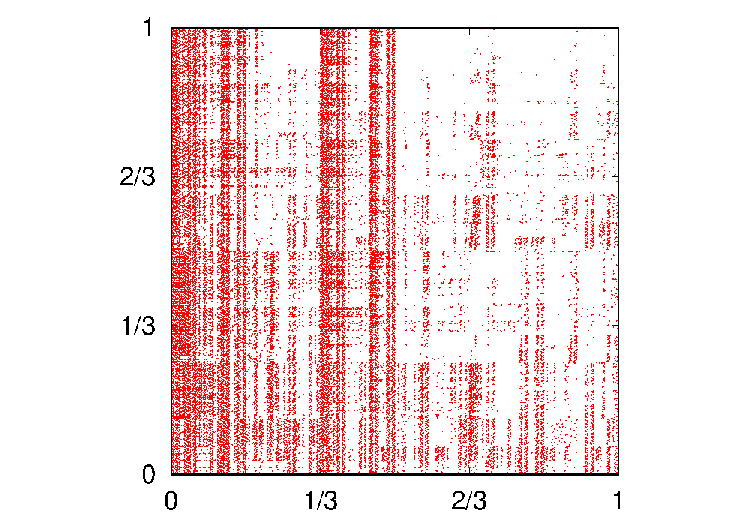}}
{\small period 13, 1-unstable}
\end{minipage}
\begin{minipage}[b]
{0.32\linewidth}
{\includegraphics[width=0.9\textwidth]{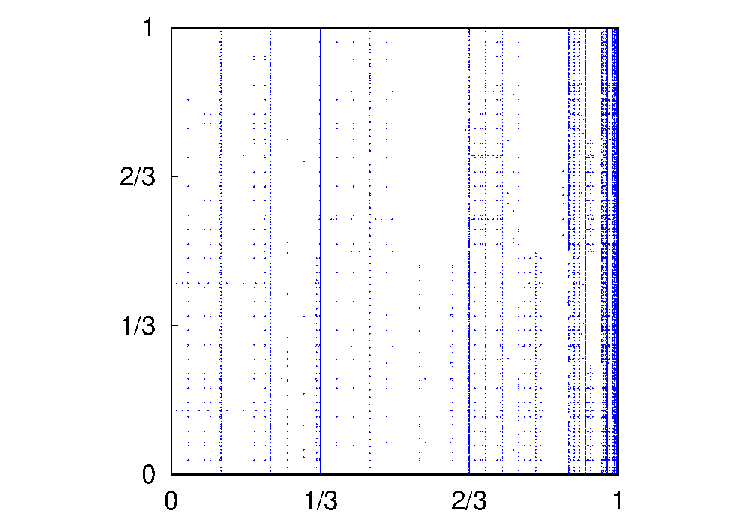}}
{\small period 11, 2-unstable}
\end{minipage}
 \begin{minipage}[b]
 {0.32\linewidth}
{\includegraphics[width=0.9\textwidth]{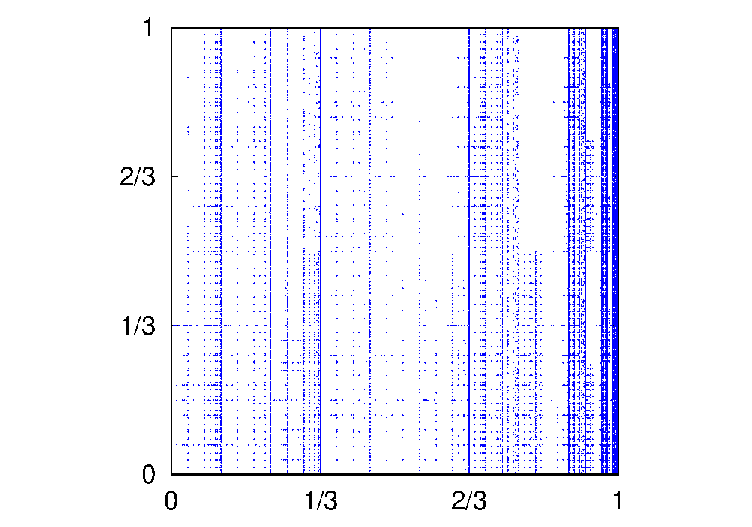}}
{\small period 12, 2-unstable}
\end{minipage}
\begin{minipage}[b]
{0.32\linewidth}
{\includegraphics[width=0.90\textwidth]{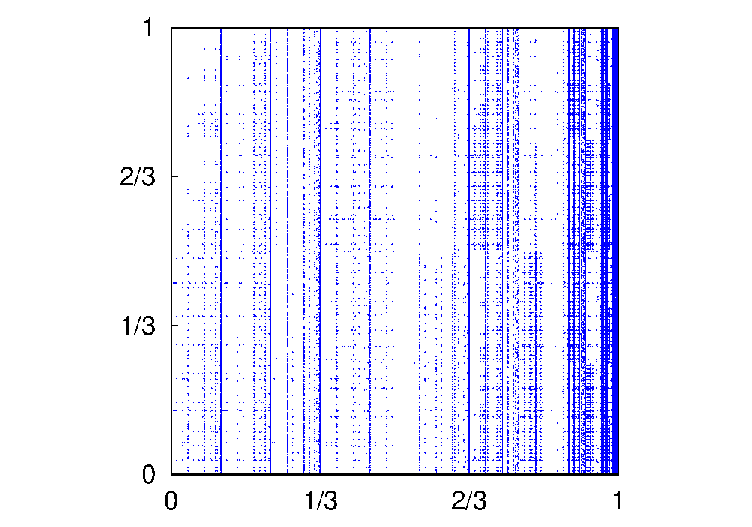}}
{\small period 13, 2-unstable}
\end{minipage}
\begin{minipage}[b]
 {0.32\linewidth}
{\includegraphics[width=0.9\textwidth]
{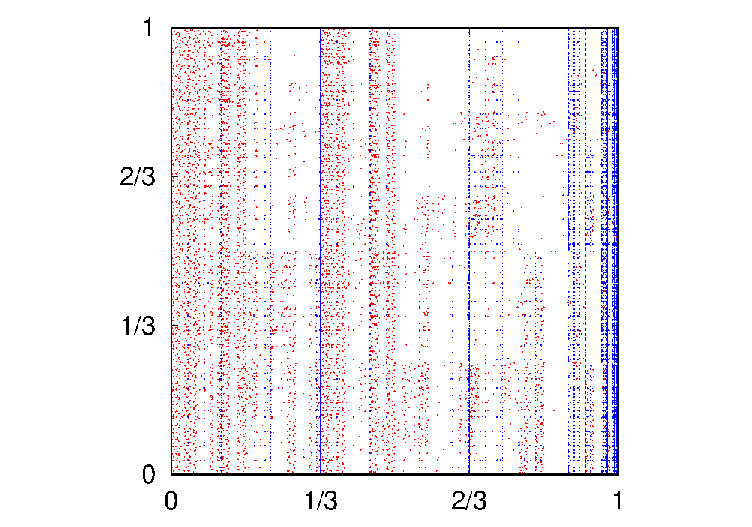}}
 {\small  period 11, 1$\&$2-unstable}
\end{minipage}
\begin{minipage}[b]
{0.32\linewidth}
{\includegraphics[width=0.9\textwidth]{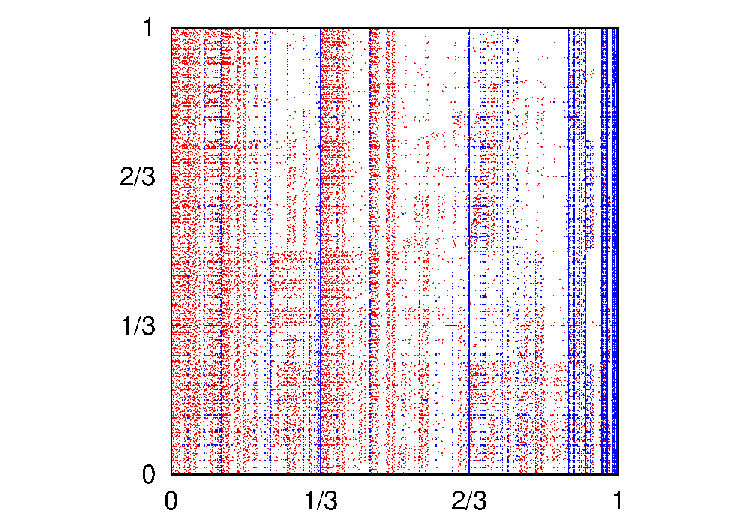}}
{\small period 12, 1$\&$2-unstable}
\end{minipage}
\begin{minipage}[b]
{0.32\linewidth}
{\includegraphics[width=0.9\textwidth]{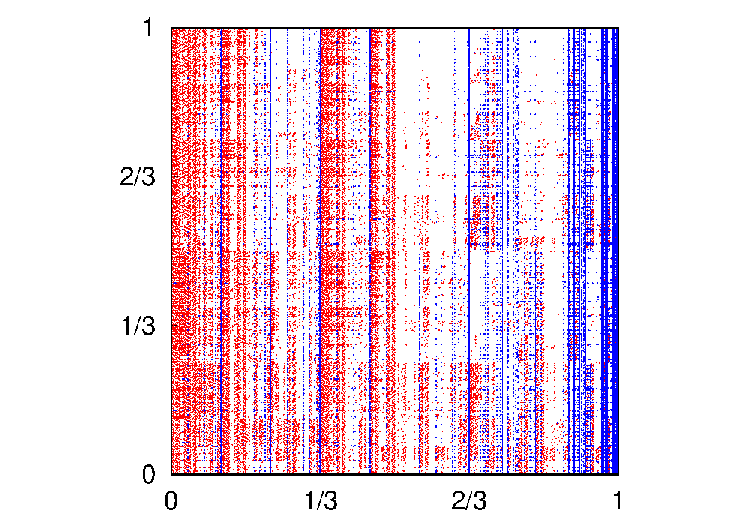}}
{\small period 13, 1$\&$2-unstable}
\end{minipage}
\caption{Part of periodic points of $f_{\frac{1}{3}}$ with $M=2$.
The first row shows $1$-unstable periodic points, the second row shows $2$-unstable periodic points,
and
the third row shows both of them.
}
\label{fig:p-orbits}
\end{figure}
\textsc{Figure}~\ref{fig:p-orbits} taken from \cite{STYY} shows partial plots of periodic points of $f_{\frac{1}{3}}$ with $M=2$ of period $11$, $12$, $13$ numerically computed by Yoshitaka Saiki. By \cite[Theorem~2.3(d)]{STYY}, the ergodic MME of $f_{\frac{1}{3}}$ obtained as the projection of
 $\mu_{\alpha}$ is the Lebesgue measure on $[0,1]^2$. So, 
$1$-unstable periodic points of $f_{\frac{1}{3}}$ are distributed according to the Lebesgue measure on $[0,1]^2$
as their periods tend to infinity. 
$2$-unstable periodic points are distributed according to the projection of $\mu_\beta$  that is singular with respect to the Lebesgue measure on $[0,1]^2$.



We hope that Theorem~\ref{theorema} sheds some light on distributions of periodic points for systems for which MMEs are not unique.
In the smooth category, 
most of such examples are partially hyperbolic systems \cite{BFT23,RT22,RRTU12}. In \cite{BFT23,RRTU12}, coexisting MMEs do not appear explicitly but appear in abstract dichotomy theorems. In \cite{RT22}, two ergodic MMEs on $\mathbb T^4$ were constructed but it is not clear how they are represented by periodic points.


\subsection{Distributions of periodic points for the Dyck shift}\label{dist-Dyck}
In order to prove Theorem~\ref{theorema}, we code  points in the maximal invariant set $\Lambda$ in \eqref{maximal-def} into sequences in the Cartesian product topological space $D^{\mathbb Z}$.
Define a {\it coding map} 
  $\pi\colon x\in \Lambda\mapsto
 (\omega_n)_{n\in\mathbb Z}\in D^{\mathbb Z}$ by
  \begin{equation}\label{code-def}x\in\bigcap_{n=-\infty }^\infty f^{-n}({\rm int}
  (\Omega_{\omega_n})).\end{equation}
  Let $\sigma$ denote the left shift acting on the subshift $\overline{\pi(\Lambda)}$: $(\sigma\omega)_n=\omega_{n+1}$ for all $n\in\mathbb Z$. 
  The coding map $\pi$ is a semiconjugacy between $f|_{\Lambda}$ and $\sigma$. 
We analyze asymptotic distributions of periodic points in the subshift $\overline{\pi(\Lambda)}$, and pull this result back to $f|_\Lambda$ to deduce Theorem~\ref{theorema}.
The subshift $\overline{\pi(\Lambda)}$ is independent of $(a,b)$ and 
   was identified in \cite{TY23} as explained below.

Let $D^*$ denote the set of finite words in $D$.
Consider the monoid with zero, with $2M$ generators in $D$ with relations 
\[\alpha_i\cdot\beta_j=\delta_{ij},\
0\cdot 0=0\ \text{ for } i,j\in\{1,\ldots,M\},\] \[\gamma\cdot 1= 1\cdot\gamma=\gamma,\
 \gamma\cdot 0=0\cdot\gamma=0\ 
\text{ for }\gamma\in D^*\cup\{ 1\},\]
where $\delta_{ij}$ denotes Kronecker's delta.
For $n\in\mathbb N$ and $\gamma_1\cdots\gamma_n\in D^*$ 
let
\[{\rm red}(\gamma_1\cdots\gamma_n)=\prod_{i=1}^n\gamma_i.\]
The subshift
\[
\Sigma_{D}=\{\omega=(\omega_i)_{i\in \mathbb Z}\in D^{\mathbb Z}\colon {\rm red}(\omega_j\cdots \omega_k)\neq0\ \text{ for all }j,k\in\mathbb Z\text{ with }j<k\}.\]
 is called {\it the Dyck shift} \cite{Kri74}. If we interpret $D$ 
as a collection of $M$ brackets, $\alpha_k$ left and $\beta_k$ right in pair, then  
$\Sigma_D$ is the subshift whose admissible words are words of legally aligned brackets.
It was proved in \cite[Theorem~1.1]{TY23} that 
\begin{equation}\label{relation-dyck}\overline{\pi(\Lambda)}=\Sigma_{D} \ \text{ for all } a, b\in \left(0,\frac{1}{M}\right).\end{equation}

Krieger \cite{Kri74} proved that
the Dyck shift has exactly two ergodic MMEs. By transferring them to $\Lambda$ 
we have obtained in \cite{TY23} the two ergodic MMEs $\mu_\alpha$, $\mu_\beta$ for the heterochaos baker map $f$. We set
\[\nu_\alpha=\mu_\alpha\circ\pi^{-1}\ \text{ and }\ \nu_\beta=\mu_\beta\circ\pi^{-1}.\] They are the two ergodic MMEs for the Dyck shift \cite{TY23}.

In order to represent $\nu_\alpha$ and $\nu_\beta$ 
by periodic points,
for each $n\in\mathbb N$ define a function $H_n\colon \Sigma_D\to\mathbb Z$ by
      \begin{equation}\label{hn-def}H_n(\omega)=\sum_{j=0}^{n-1} \sum_{k=1}^{M}(\delta_{{\alpha_k},\omega_j}-\delta_{\beta_k,\omega_j}).\end{equation}
      Note that $H_n(\omega)$ equals
      the difference of the number of symbols in $D_\alpha$ and that in $D_\beta$ in the sequence $\omega_0\omega_1\cdots \omega_{n-1}$.
We decompose ${\rm Per}_{n}(\sigma)$ into the following three subsets:
\[\begin{split}{\rm Per}_{0,n}(\sigma)&=\{\omega\in {\rm Per}_n(\sigma) \colon H_n(\omega)=0\};\\
{\rm Per}_{\alpha,n}(\sigma)&=\{\omega\in {\rm Per}_n(\sigma) \colon H_n(\omega)>0\};\\
{\rm Per}_{\beta,n}(\sigma)&=\{\omega\in {\rm Per}_n(\sigma)\colon H_n(\omega)<0\}.\end{split}\]
We exclude from further consideration all the periodic points in $\bigcup_{n\in\mathbb N}{\rm Per}_{0,n}(\sigma)$.
A zeta function defined by these periodic points was considered in \cite{Kel91}.
In \cite{HI05}, periodic points in $\bigcup_{n\in\mathbb N}{\rm Per}_{\alpha,n}(\sigma)$ (resp. $\bigcup_{n\in\mathbb N}{\rm Per}_{\beta,n}(\sigma)$) are said to have negative (resp. positive) multipliers.

By virtue of the connection \eqref{relation-dyck} between the heterochaos baker maps and the Dyck shift,  Theorem~\ref{theorema} follows from the next theorem on the Dyck shift.
 \begin{thm}\label{theoremb}
For any continuous function $\phi\colon\Sigma_D\to\mathbb R$ we have
\[\lim_{n\to\infty}\frac{\sum_{\omega\in{\rm Per}_{\alpha,n}(\sigma)}\phi(\omega)}{\#{\rm Per}_{\alpha,n}(\sigma)}=\int\phi{\rm d}\nu_{\alpha}\ \text{ and }\
\lim_{n\to\infty}\frac{\sum_{\omega\in{\rm Per}_{\beta,n}(\sigma)}\phi(\omega)}{\#{\rm Per}_{\beta,n}(\sigma)}=\int\phi{\rm d}\nu_{\beta},\]
and
\[\lim_{n\to\infty}\frac{\sum_{\omega\in{\rm Per}_{\alpha,n}(\sigma)\cup {\rm Per}_{\beta,n}(\sigma)  }\phi(\omega)}{\#({\rm Per}_{\alpha,n}(f)\cup {\rm Per}_{\beta,n}(f) )}=\frac{1}{2}\int\phi{\rm d}\nu_{\alpha}+\frac{1}{2}\int\phi{\rm d}\nu_{\beta}.\]
\end{thm}

We hope that Theorem~\ref{theoremb} sheds some light on distributions of periodic points of general subshifts for which the MMEs are not unique.
For such examples other than the Dyck shift, see e.g., \cite{GK18,Hay13,KOR16,Pav16} and the references therein. 
Little is known on how the coexisting MMEs can be represented by periodic points
in these examples.

Krieger \cite{Kri74} proved that there exist two different full shifts $\Sigma_\alpha$, $\Sigma_\beta$ on $M+1$ symbols, shift-invariant Borel sets $K_\gamma\subset\Sigma_\gamma$
($\gamma\in\{\alpha,\beta\}$) and homeomorphisms $\psi_\gamma\colon K_\gamma\to\Sigma_D$ that commute with the left shifts. The Bernoulli measure $\xi_\gamma$ on $\Sigma_\gamma$ associated with the probability vector $(\frac{1}{M+1},\ldots,\frac{1}{M+1})$ gives measure $1$ to $K_\gamma$, and satisfies $\nu_\gamma=\xi_\gamma\circ\psi_\gamma^{-1}$. Moreover,
$\psi_{\gamma}^{-1}({\rm Per}_{\gamma,n}(\sigma))$ is contained in $K_\gamma$ for all $n\in\mathbb N$. We show that the set of these 
periodic points embedded into $\Sigma_\gamma$ are distributed according to $\xi_\gamma$ in the weak* topology on $M(K_\gamma)$ as their periods tends to infinity. 
To show this convergence,
Bowen's argument \cite{Bow71} cannot be used directly 
since $\psi_{\gamma}^{-1}({\rm Per}_{\gamma,n}(\sigma))$ is not a separated set.
We establish the convergence by means of a large deviations approach of Kifer \cite{Kif94}.  
Finally we transfer this convergence via $\psi_\gamma$ back to the Dyck shift space.
By the symmetry in the Dyck shift, 
the last equality in Theorem~\ref{theoremb}
follows from the first two.

The rest of this paper consists of two sections. In Section~2 we collect and prove preliminary results on the Dyck shift needed for the proof of Theorem~\ref{theoremb}.
 In Section~3 we prove Theorem~\ref{theoremb} and then
Theorem~\ref{theorema}.
 
\section{Preliminaries on the Dyck shift}

Throughout this section,
 let $M\geq2$ be an integer and let $\Sigma_D$ be the Dyck shift on $2M$ symbols. 
 After introducing basic notations 
 in Section~\ref{notation},  we delve into the structure of $\Sigma_D$ in Section~\ref{iso-sec} and Section~\ref{DM-str}. In Section~\ref{MME-sec} we outline the construction of the two ergodic MMEs by Krieger \cite{Kri74}.
In Section~\ref{count-sec} we estimate the number of periodic points using results of Hamachi and Inoue \cite{HI05}.
\subsection{Notation}\label{notation}Let $S$ be a non-empty finite discrete set, called an alphabet, and let $S^{\mathbb Z}$ denote 
the two-sided Cartesian product topological space of $S$, called the {\it full shift}.
The left shift acts continuously on $S^{\mathbb Z}$.
  A {\it subshift} over the alphabet $S$ is a shift-invariant closed subset of $S^{\mathbb Z}$. 
  For a subshift $\Sigma$ over $S$ and for $j\in\mathbb Z$, $n\in\mathbb N$, $\theta=\theta_1\cdots\theta_n\in S^n$, define 
 \[\Sigma(j;\theta)=\{(\omega_i)_{i\in\mathbb Z}\in\Sigma\colon \omega_i=\theta_{i-j+1}\text{ for }i=j,\ldots, j+n-1\}.\]

  We introduce two full shifts over different alphabets consisting of $M+1$ symbols:
 \[\Sigma_\alpha=(D_\alpha\cup\{\beta\})^{\mathbb Z}\quad\text{and}\quad\Sigma_\beta=(\{\alpha\}\cup D_\beta)^{\mathbb Z}.\]
 Let $\sigma_\alpha$, $\sigma_\beta$ denote the left shifts acting on $\Sigma_\alpha$, $\Sigma_\beta$ respectively.
Let $\xi_\alpha$, $\xi_\beta$ denote the 
  Bernoulli measures on 
$\Sigma_{\alpha}$,
$\Sigma_{\beta}$ respectively associated with the probability vector  $(\frac{1}{M+1},\ldots,\frac{1}{M+1})$.

We work on three subshifts $\Sigma_D$, $\Sigma_\alpha$, $\Sigma_\beta$, and Borel probability measures on them.
For readability, 
we use the letters  
$\nu$ and $\xi$ (with subscripts) to denote elements of  $M(\Sigma_D)$ and $M(\Sigma_\gamma)$ $(\gamma=\alpha,\beta)$ respectively.
The letters $\omega$ and $\zeta$ are used to denote points in $\Sigma_D$ and $\Sigma_\gamma$ $(\gamma=\alpha,\beta)$ respectively. 
Let $C(\Sigma_\alpha)$,  $C(\Sigma_\beta)$ denote the spaces of real-valued continuous functions on  $\Sigma_\alpha$, $\Sigma_\beta$ respectively endowed with the supremum norm.

\subsection{Classification of ergodic measures }\label{iso-sec}
Similarly to the definition \eqref{hn-def}, for each $i\in\mathbb Z$ we define a function $H_i\colon \Sigma_D\to\mathbb Z$ by
      \[H_i(\omega)=\begin{cases}\sum_{j=0}^{i-1} \sum_{k=1}^{M}(\delta_{{\alpha_k},\omega_j}-\delta_{\beta_k,\omega_j})&\text{ for }  i\geq1,\\\sum_{j=i}^{-1} \sum_{k=1}^{M}(\delta_{{\beta_k},\omega_j}-\delta_{\alpha_k,\omega_j})&\text{ for } i\leq -1,\\    0&\text{ for }i=0.\end{cases}\]
For $i$, $j\in\mathbb Z$ define \[\{H_i=H_j\}=\{\omega\in\Sigma_D\colon H_i(\omega)=H_j(\omega)\}.\]
We introduce three pairwise disjoint shift invariant Borel sets:
    \[\label{3-sets}\begin{split}A_0&=\bigcap_{i=-\infty }^{\infty}\left(\left(\bigcup_{j=1}^\infty\{ H_{i+j}=H_i\}\right)\cap\left(\bigcup_{j=1}^\infty\{ H_{i-j}=H_i\}\right)\right);\\
A_\alpha&=\left\{\omega\in\Sigma_D\colon
\lim_{i\to\infty}H_i(\omega)
=\infty\ \text{ and } \ \lim_{i\to-\infty}H_i(\omega)=-\infty\right\};\\
A_\beta&=\left\{\omega\in\Sigma_D\colon
\lim_{i\to\infty}H_i(\omega)=-\infty\ \text{ and } \
\lim_{i\to-\infty}H_i(\omega)=\infty\right\}.\end{split}\]
Note that all the three sets are dense in $\Sigma_D$.

\begin{lemma}[\cite{Kri74}, pp.102--103]
\label{trichotomy}
If $\nu\in M(\Sigma_D,\sigma)$ is ergodic, then either $\nu(A_0)=1$, 
$\nu(A_\alpha)=1$ or $\nu(A_\beta)=1$.
\end{lemma}

\subsection{Construction of Borel embeddings of the full shift}\label{DM-str}
Under the notation in Section~\ref{notation}, we introduce two shift-invariant Borel sets of $\Sigma_D$:
\[\begin{split}B_\alpha&=\bigcap_{i=-\infty}^\infty\bigcup_{k=1}^M\left(\Sigma_{D}(i;\alpha_k)\cup\left(\Sigma_{D}(i;\beta_k)\cap\bigcup_{j=1}^\infty\{ H_{i-j+1}=H_{i+1}\}\right)\right);\\
B_\beta&=\bigcap_{i=-\infty}^\infty\bigcup_{k=1}^M\left(\Sigma_{D}(i;\beta_k)
\cup\left(\Sigma_{D}(i;\alpha_k)\cap \bigcup_{j=1}^\infty\{H_{i+j}=H_i\}\right)\right).\end{split}\]
 The set $B_\alpha$ (resp. $B_\beta$)
is precisely the set of sequences in $\Sigma_D$ such that any right (resp. left) bracket in the sequence is closed.
One can check that
\begin{equation}\label{subset-AB}A_0\cup A_\alpha\subset B_\alpha\ \text{ and }\ A_0\cup A_\beta\subset B_\beta.\end{equation}

Define 
 $\phi_\alpha\colon \Sigma_D\to \Sigma_\alpha$ 
by
\[(\phi_\alpha(\omega))_i=\begin{cases}
  \beta&\text{ if }\omega_i\in D_\beta,\\
 \omega_i &\text{ otherwise.}
\end{cases}\]
In other words, $\phi_\alpha(\omega)$ is obtained by replacing all $\beta_k$, $k\in\{1,\ldots,M\}$ in $\omega$ by $\beta$.
Clearly $\phi_\alpha$ is continuous.
Similarly, define $\phi_\beta\colon \Sigma_D\to \Sigma_\beta$ by
\[(\phi_\beta(\omega))_i=\begin{cases}
    \alpha&\text{ if }\omega_i\in D_\alpha,\\
    \omega_i&\text{ otherwise.}
\end{cases}\]
In other words, $\phi_\beta(\omega)$ is obtained by replacing all $\alpha_k$, $k\in\{1,\ldots,M\}$ in $\omega$ by $\alpha$.
Clearly $\phi_\beta$ is continuous too. We set
\[K_\alpha=\phi_\alpha(B_\alpha)\ \text{ and }\ K_\beta=\phi_\beta(B_\beta).\]

    For each $i\in\mathbb Z$ define $ H_{\alpha,i}\colon \Sigma_\alpha\to\mathbb Z$ by       
      \[\begin{split}H_{\alpha,i}(y)&=\begin{cases}\sum_{j=0}^{i-1} \sum_{k=1}^{M}(\delta_{{\alpha_k},y_j}-\delta_{\beta,y_j})&\text{ for }  i\geq1,\\\sum_{j=i}^{-1} \sum_{k=1}^{M}(\delta_{{\beta},y_j}-\delta_{\alpha_k,y_j})&\text{ for } i\leq -1,\\
      0&\text{ for }i=0.\end{cases}\end{split}\]
We now define $\psi_\alpha\colon K_\alpha\to D^\mathbb Z$ by
\[(\psi_\alpha(y))_i=\begin{cases}
  \beta_k&\text{ if }y_i=\beta,\ y_{s_\alpha(i,y)}=
  \alpha_k,\ k\in\{1,\ldots,M\},\\
   y_i&\text{ otherwise,}
\end{cases}\]
where \[s_\alpha(i,y)=\max\{j<i+1\colon  H_{\alpha,j}(y)= H_{\alpha,i+1}(y)\}.\]
Clearly $\psi_\alpha$ is continuous. Similarly, for each $i\in\mathbb Z$ we define 
$H_{\beta,i}\colon \Sigma_\beta\to\mathbb Z$ by 
      \[\begin{split}
            H_{\beta,i}(y)&=\begin{cases}\sum_{j=0}^{i-1} \sum_{k=1}^{M}(\delta_{{\alpha},y_j}-\delta_{\beta_k,y_j})&\text{ for }  i\geq1,\\\sum_{j=i}^{-1} \sum_{k=1}^{M}(\delta_{{\beta_k},y_j}-\delta_{\alpha,y_j})&\text{ for } i\leq -1,\\
      0&\text{ for }i=0.\end{cases}\end{split}\]
We also define $\psi_\beta\colon K_\beta\to D^\mathbb Z$ by
\[(\psi_\beta(y))_i=\begin{cases}
   \alpha_k&\text{ if }y_i=\alpha,\ y_{s_\beta(i,y)}=\beta_k,\ k\in\{1,\ldots,M\},\\
    y_i&\text{ otherwise, }
\end{cases}\]
where \[s_\beta(i,y)=\min\{j>i\colon  H_{\beta,j}(y)= H_{\beta,i}(y)\}.\]
Clearly $\psi_\beta$ is continuous too.

 \begin{lemma}[\cite{Kri74}, Section~4]\label{include-lem}
 Let $\gamma\in\{\alpha,\beta\}$.
\begin{itemize}
\item[(a)] 
$\psi_\gamma(K_\gamma)=B_\gamma$, and $\psi_\gamma$ is a homeomorphism whose inverse is $\phi_\gamma|_{B_\gamma}$.
\item[(b)] 
$\phi_\gamma\circ\sigma|_{B_\gamma}=\sigma_\gamma\circ\phi_\gamma|_{B_\gamma}$ and $\sigma^{-1}\circ\psi_\gamma=\psi_\gamma\circ\sigma_\gamma^{-1}|_{K_\gamma }$.
\end{itemize}
\end{lemma}

Elements of $M(\Sigma_\gamma)$ that give measure $1$
to $K_\gamma$ can be transported via $\psi_\gamma$ 
to elements of $M(\Sigma_D)$.
 The lemma below gives a sufficient condition for ergodic elements of $M(\Sigma_\gamma,\sigma_\gamma)$ to be transported
to elements of $M(\Sigma_D)$.
Let $\mathbbm{1}_{(\cdot)}$ denote the indicator function for a set.

\begin{lemma}[\cite{STYY}, Lemma~3.3]\label{gyak-lem}
\ 
\begin{itemize}
    \item[(a)] If $\xi\in M(\Sigma_\alpha,\sigma_\alpha)$ is ergodic and $\int\mathbbm{1}_{\Sigma_{\alpha}(0;\beta)}{\rm d}\xi<\frac{1}{2}$ then $\xi(K_\alpha)=1$.
    \item[(b)] If $\xi\in M(\Sigma_\beta,\sigma_\beta)$ is ergodic and $\int\mathbbm{1}_{\Sigma_{\beta}(0;\alpha)}{\rm d}\xi<\frac{1}{2}$ then $\xi(K_\beta)=1$.
\end{itemize}
\end{lemma}

\begin{lemma}\label{dense-lem}
 Let $\gamma\in\{\alpha,\beta\}$.
 $K_\gamma$ is a dense subset of $\Sigma_\gamma$.\end{lemma}
\begin{proof}
Clearly the Bernoulli measure $\xi_\alpha$ on $\Sigma_\alpha$
satisfies
$\int\mathbbm{1}_{\Sigma_{\alpha}(0;\beta)}{\rm d}\xi_\alpha<\frac{1}{2}$.
By Lemma~\ref{gyak-lem}(a)
 we have $\xi_\alpha(K_\alpha)=1$. Since $\xi_\alpha$ charges any nonempty open subset of $\Sigma_\alpha$, it follows that $K_\alpha$ is a dense subset of $\Sigma_\alpha$. A proof of the denseness of $K_\beta$ in $\Sigma_\beta$ is completely analogous.\end{proof}

\subsection{The ergodic MMEs for the Dyck shift}\label{MME-sec}


As in the proof of Lemma~\ref{dense-lem}, we have
$\xi_\alpha(K_\alpha)=1=\xi_\beta(K_\beta)$. 
Hence, the measures  \begin{equation}\label{ergmme}\nu_{\alpha}=\xi_\alpha\circ\psi_\alpha^{-1}\ \text{ and } \ \nu_\beta=\xi_\beta\circ\psi_\beta^{-1}\end{equation}
are 
Bernoulli 
of entropy $\log(M+1)$.
By Lemma~\ref{dense-lem}, they charge any non-empty open subset of $\Sigma_D$.
 From direct calculations based on \eqref{ergmme}, 
we deduce the following identities for $k=1,\ldots,M$:
 \begin{equation}\begin{split}\label{balance2}
 \nu_\alpha
\left(\Sigma_D(0;\alpha_k)\right)&=\nu_\alpha
\left(\sum_{j=1}^M\Sigma_D(0;\beta_j)\right)=\frac{1}{M+1};\\
\nu_\beta
\left(\sum_{j=1}^M\Sigma_D(0;\alpha_j)\right)&=
\nu_\beta\left(\Sigma_{D}(0;\beta_k)\right)=\frac{1}{M+1 }.\end{split}\end{equation}
In particular, $\nu_\alpha\neq\nu_\beta$ holds.
The ergodicity of $\xi_\alpha$, $\xi_\beta$ and \eqref{balance2}  altogether imply
 \begin{equation}\label{measure-nu}\nu_\alpha(A_\alpha )=1\ \text{ and }\ 
 \nu_\beta(A_\beta )=1.\end{equation}
 From \eqref{subset-AB}, Lemma~\ref{trichotomy} and \eqref{measure-nu}
 it follows that $\nu_\alpha$, $\nu_\beta$ are ergodic MMEs with entropy $\log(M+1)$, and that there is no other ergodic MME. We have outlined the proof of the following theorem due to Krieger \cite{Kri74}.

 \begin{thm}[\cite{Kri74}]\label{k-thm}
There exist exactly two shift invariant ergodic Borel probability measures of maximal entropy $\log(M+1)$ for $(\Sigma_D,\sigma)$. 
They are 
Bernoulli and 
 charge any non-empty open subset of  $\Sigma_D$.
\end{thm}
\subsection{Estimate of the number of periodic points}\label{count-sec}Hamachi and Inoue \cite{HI05} obtained 
 exact formulas on numbers of periodic points of the Dyck shift. For our purpose we prove the next lemma using 
 results in \cite{HI05}.
\begin{lemma}\label{per-lem} Let $\gamma\in\{\alpha,\beta\}$. For all sufficiently large $n\geq1$ we have \[\frac{1}{3}(M+1)^n\leq\#{\rm Per}_{\gamma,n}(\sigma )<(M+1)^n.\]\end{lemma}
\begin{proof}
By the symmetry in the Dyck shift, we have $\#{\rm Per}_{\alpha,n}(\sigma)=\#{\rm Per}_{\beta,n}(\sigma )$ for all $n\in\mathbb N$. Hence,
for each $\gamma\in\{\alpha,\beta\}$ 
we have \begin{equation}\label{per-eq10}\#{\rm Per}_{\gamma,n}(\sigma)=\frac{1}{2}(\#{\rm Per}_{n}(\sigma)-\#{\rm Per}_{0,n}(\sigma))\ \text{ for all }n\in\mathbb N.\end{equation}
A direct calculation shows
\begin{equation}\label{per-eq0}\#{\rm Per}_{0,n}(\sigma )=\begin{cases}\vspace{1mm}\displaystyle{\binom{n}{n/2 }M^{n/2}}&\ \text{ if $n$ is even,}\\ 0&\ \text{ if $n$ is odd.}\end{cases}\end{equation} 
Substituting \eqref{per-eq0} and the formula for
 $\#{\rm Per}_{n}(\sigma )$ in \cite[Lemma~2.5]{HI05} into the right-hand side of \eqref{per-eq10}, we get
\begin{equation}\label{per-eq1}\#{\rm Per}_{\gamma,n}(\sigma)=(M+1)^n-\sum_{i=0}^{\lfloor n/2\rfloor}\binom{n}{i}M^i,\end{equation}
where $\lfloor  s\rfloor$ for $s>0$ denotes the largest integer not exceeding $s$.
Hence the desired upper bound holds.

We have
\begin{equation}\label{per-eq2}\sum_{i=0}^{\lfloor n/2\rfloor}\binom{n}{i}M^i=\frac{1}{2}\left((M+1)^n+\binom{n}{\lfloor n/2\rfloor}M^{\lfloor n/2\rfloor}\right).\end{equation}
By Stirling's formula for factorials, for all sufficiently large $n$ we have
\begin{equation}\label{per-eq3}\binom{n}{\lfloor n/2\rfloor}M^{\lfloor n/2\rfloor}\leq\frac{1}{\sqrt{n}}(2\sqrt{M})^n.\end{equation}
Plugging \eqref{per-eq2}, \eqref{per-eq3} into the right-hand side of \eqref{per-eq1} and then rearranging the result yields the desired lower bound for all sufficiently large $n\geq1$. 
\end{proof}

\section{Distributions of periodic points}
To complete the proofs of the main results, in Section~\ref{restricted} 
we introduce a sequence of Borel probability measures on $M(\Sigma_\gamma)$ constructed from the periodic points in $\bigcup_{n\in\mathbb N}\phi_\gamma({\rm Per}_{\gamma,n}(\sigma))$, and
prove a large deviations upper bound using the result of Kifer \cite{Kif94}. 
In Section~\ref{str-code}
we analyze the structure of the coding map.
We prove 
Theorem~\ref{theoremb} in Section~\ref{pfthmb},
 and then prove
 Theorem~\ref{theorema}
 in Section~\ref{pfthma}.

\subsection{A large deviations upper bound}\label{restricted}
For each $\gamma\in\{\alpha,\beta\}$ and $n\in\mathbb N$, define 
$\widetilde\xi_{\gamma,n}\in M(M(\Sigma_\gamma))$ by
\[\widetilde\xi_{\gamma,n}=\frac{\sum_{\zeta\in \phi_\gamma({\rm Per}_{\gamma,n}(\sigma ))}\delta_{V_{n}(\sigma_\gamma,\zeta)}}{\#{\rm Per}_{\gamma,n}(\sigma ) },\]
where $V_n(\sigma_\gamma,\zeta)=n^{-1}(\delta_\zeta+\cdots+\delta_{\sigma_\gamma^{n-1}\zeta})\in M(\Sigma_\gamma)$, and $\delta_{V_{n}(\sigma_\gamma,\zeta)}\in M(M(\Sigma_\gamma))$ denotes the unit point mass at $V_n(\sigma_\gamma,\zeta)$. 
Define $J_\gamma\colon M(\Sigma_\gamma)\to [0,\infty]$
 by
\begin{equation}\label{free-def}J_\gamma(\xi)=\begin{cases}\log(M+1)-h(\xi,\sigma_\gamma)\ &\text{ if $\xi\in M(\Sigma_\gamma,\sigma_\gamma)$},\\
\infty&\text{ otherwise}.\end{cases}\end{equation}
Since the entropy function on $M(\Sigma_\gamma,\sigma_\gamma)$
is upper semicontinuous, $J_\gamma$ is lower semicontinuous. Note that $J_\gamma(\xi)=0$ if and only if $\xi=\xi_\gamma$.
\begin{lemma}\label{LDP-rest} 
Let $\gamma\in\{\alpha,\beta\}$. For any closed set $\mathcal C$ of $M(\Sigma_\gamma)$, 
\begin{equation*}
\limsup_{n\to\infty}\frac{1}{n}\log \widetilde\xi_{\gamma,n}(\mathcal C)\leq-\inf_{\mathcal C}J_\gamma,\end{equation*}
where $\inf\emptyset=\infty$ and $\log0=-\infty$.\end{lemma}

\begin{proof}
For any closed subset $\mathcal C$ of $M(\Sigma_\gamma)$, we have 
\[\begin{split}\widetilde\xi_{\gamma,n}(\mathcal C)&=\frac{\#\{\zeta\in  \phi_\gamma({\rm Per}_{\gamma,n}(\sigma )) \colon V_n(\sigma_\gamma,\zeta)\in\mathcal C\}}{\#{\rm Per}_{\gamma,n}(\sigma)}\\
&\leq\frac{\#\{\zeta\in  {\rm Per}_{n}(\sigma_\gamma ) \colon V_n(\sigma_\gamma,\zeta)\in\mathcal C\}}{\#{\rm Per}_{\gamma,n}(\sigma)}.\end{split}\]
Taking logs of both sides, dividing the result by $n$ and letting $n\to\infty$ yields
\[\begin{split}&\limsup_{n\to\infty}\frac{1}{n}\log\widetilde\xi_{\gamma,n}(\mathcal C)
\\
&\leq \limsup_{n\to\infty}\frac{1}{n}\log\#\{\zeta\in  {\rm Per}_{n}(\sigma_\gamma ) \colon V_n(\sigma_\gamma,\zeta)\in\mathcal C\}-\liminf_{n\to\infty}\frac{1}{n}\log\#{\rm Per}_{\gamma,n}(\sigma)\\
&\leq\sup \{h(\xi,\sigma_\gamma)\colon\xi\in M(\Sigma_\gamma,\sigma_\gamma)\cap\mathcal C\}-\log(M+1)\\
&\leq-\inf_{\mathcal C}J_\gamma,\end{split}\]
as required. The last inequality follows from 
 \cite[Theorem~2.1]{Kif94} and
Lemma~\ref{per-lem}.
\end{proof}

\subsection{Structure of the coding map}\label{str-code}
The coding map $\pi\colon\Lambda\to\Sigma_D$ introduced in \eqref{code-def} is not injective. 
In order to clarify where the preimage of $\pi$ is a singleton, we  
consider the set
\[A_{\alpha,\beta}=\left\{\omega\in\Sigma_D\colon \liminf_{i\to\infty}H_i(\omega)=-\infty\ \text{ or } \ \liminf_{i\to-\infty}H_i(\omega)=-\infty\right\}.\]
Note that $A_\alpha\cup A_\beta\subset A_{\alpha,\beta}$. 

 Let $\omega\in\Sigma_D$.
 For each $i\in\mathbb Z$ define 
      \[K_i(\omega)=\begin{cases}
      \bigcap_{j=0}^{i-1}f^{-j}(\Omega_{\omega_j})&\text{ for }  i\geq1,\\
      \bigcap_{j=-i+1}^{0}f^{-j}(\Omega_{\omega_j})&\text{ for } i\leq -1,\\
      [0,1]^3&\text{ for }i=0.\end{cases}\]
 Clearly we have $\pi^{-1}(\omega)\subset\bigcap_{i=-\infty }^\infty K_i(\omega)$.

\begin{lemma}[\cite{STYY} Lemma~3.7]\label{factor-prop}
 If $\omega\in A_{\alpha,\beta}$ then
$\bigcap_{i=-\infty }^\infty K_i(\omega)$ is a singleton.
If moreover $\omega\in\pi(\Lambda)$, then
 $\pi^{-1}(\omega)$ is a singleton.
 \end{lemma}

\subsection{Proof of Theorem~\ref{theoremb}}\label{pfthmb}
For each $\gamma\in\{\alpha,\beta\}$ and $n\in\mathbb N$, define
\[\nu_{\gamma,n}=\frac{\sum_{\omega\in {\rm Per}_{\gamma,n}(\sigma )}\delta_\omega}{\#{\rm Per}_{\gamma,n}(\sigma ) }\in M(\Sigma_D,\sigma)        \ \text{ and } \      \xi_{\gamma,n}=\frac{\sum_{
\zeta\in\phi_\gamma( {\rm Per}_{\gamma,n}(\sigma ))}\delta_\zeta}{\#{\rm Per}_{\gamma,n}(\sigma ) }\in M(\Sigma_\gamma,\sigma_\gamma).\]
Note that the first (resp. second) convergence in Theorem~\ref{theoremb} is equivalent to the convergence of $\{\nu_{\alpha,n}\}$ to $\nu_\alpha$ (resp. 
$\{\nu_{\beta,n}\}$ to $\nu_\beta$) in the weak* topology on $M(\Sigma_D)$.

We define a continuous map
$\Pi_\gamma\colon M(M(\Sigma_\gamma))\to M(\Sigma_\gamma)$ as follows.
Let $\widetilde\xi\in M(M(\Sigma_\gamma))$.
Consider the positive normalized bounded linear functional on $C(\Sigma_\gamma)$ given by
\[\phi\in C(\Sigma_\gamma)\mapsto\int_{M(\Sigma_\gamma)} \left(\int\phi{\rm d\xi}\right){\rm d}\widetilde\xi(\xi).\]
In view of Riesz's representation theorem, we define $\Pi_\gamma(\widetilde\xi)$ to be the unique element of $M(\Sigma_\gamma)$ such that
\[\int\phi{\rm d}\Pi_\gamma(\widetilde\xi)=\int_{M(\Sigma_\gamma)} \left(\int\phi{\rm d\xi}\right){\rm d}\widetilde\xi(\xi)\ \text{ for any $\phi\in C(\Sigma_\gamma)$.}\]
Clearly $\Pi_\gamma$ is continuous, satisfies 
$\Pi(\widetilde\xi_{\gamma,n})=\xi_{\gamma,n}$
and $\Pi_\gamma(\delta_\xi)=\xi$ for any $\xi\in M(\Sigma_\gamma)$ where $\delta_\xi$ denotes the unit point mass at $\xi$.

From Lemma~\ref{LDP-rest}
it follows that
 $\{\widetilde\xi_{\gamma,n}\}$ converges to $\delta_{\xi_\gamma}$. Since $\Pi_\gamma$ is continuous, $\{\xi_{\gamma,n}\}$ converges to $\xi_\gamma$ in the weak* topology on $M(\Sigma_\gamma)$. By the lemma below and $\xi_{\gamma,n}(K_\gamma)=1=\xi_{\gamma}(K_\gamma)$,
 $\{\xi_{\gamma,n}\}$ converges to $\xi_\gamma$  in the weak* topology on $M(K_\gamma)$.
 \begin{lemma}\label{top-lem}
 Let $\gamma\in\{\alpha,\beta\}$. The weak* topology on $M(K_\gamma)$ coincides with the relative topology inherited from the weak* topology on $M(\Sigma_\gamma)$. \end{lemma}
 Define $\psi_{\gamma}^*\colon M(K_\gamma)\to M(\Sigma_D)$ by $\psi_{\gamma}^*(\xi)=\xi\circ\psi_\gamma^{-1}$ for $\xi\in M(K_\gamma)$.
 Then $\psi_\gamma^*$ is continuous, satisfies
$\psi_\gamma^*(\xi_{\gamma,n})=\nu_{\gamma,n}$ and 
$\psi_\gamma^*(\xi_{\gamma})=\nu_{\gamma}$.
Hence, $\{\nu_{\gamma,n}\}$ converges to $\nu_\gamma$ in the weak* topology on $M(\Sigma_D)$ as required in Theorem~\ref{theoremb}.
The last convergence in Theorem~\ref{theoremb} follows from the first two and $\#{\rm Per}_{\alpha,n}(\sigma)=\#{\rm Per}_{\beta,n}(\sigma )$ for all $n\in\mathbb N$.

It is left to prove Lemma~\ref{top-lem}.
Let $C_u(K_\gamma)$ denote the set of real-valued, bounded  uniformly continuous functions on $K_\gamma$.
Recall that the weak* topology of $M(K_\gamma)$ is the coarsest topology that makes the function $\xi\in M(K_\gamma)\mapsto\int \phi{\rm d}\xi$ continuous for any $\phi\in C_u(K_\gamma)$.
The restriction of any element of
$C(\Sigma_\gamma)$ to $K_\gamma$ defines an element of $C_u(K_\gamma)$.
Since $K_\gamma$ is dense in $\Sigma_\gamma$
by Lemma~\ref{dense-lem}, any element of $C_u(K_\gamma)$ 
can be extended uniquely to an element of $C(\Sigma_\gamma)$. 
 It follows that
 $\phi\in C(\Sigma_\gamma)\mapsto \phi|_{K_\gamma}\in C_u(K_\gamma)$ is a bijection. Hence the assertion of Lemma~\ref{top-lem} holds.
\qed

\subsection{Proof of Theorem~\ref{theorema}}\label{pfthma}
Let $a,b\in(0,\frac{1}{M})$ and write $f=f_{a,b}$.
For each $n\in\mathbb N$, define
\[
\mu_{\gamma,n}=\frac{\sum_{x\in {\rm Per}_{\gamma,n}(f)}\delta_x}{\#{\rm Per}_{\gamma,n}(f) }\in M([0,1]^3,f).
\]
Note that the first (resp. second) convergence in Theorem~\ref{theorema} is equivalent to the convergence of $\{\mu_{\alpha,n}\}$ to $\mu_\alpha$ (resp. 
$\{\mu_{\beta,n}\}$ to $\mu_\beta$) in the weak* topology on $M([0,1]^3)$.

By Theorem~\ref{theoremb},
$\{\nu_{\gamma,n}\}$ converges to $\nu_\gamma$
in the weak* topology on $M(\Sigma_\gamma)$.
 We have $\nu_{\gamma,n}(A_\gamma)=1$.
From Birkhoff's ergodic theorem and Lemma~\ref{trichotomy}, we have $\nu_{\gamma}(A_\gamma)=1$.
Since $A_\gamma$ is a dense subset of $\Sigma_D$, any bounded uniformly continuous real-valued function on $A_\gamma$ can be extended uniquely to a continuous function on $\Sigma_D$. So,
$\{\nu_{\gamma,n}\}$ converges to $\nu_\gamma$ in the weak* topology on $A_\gamma$.

Put $M'(A_\gamma)=\{\nu\in M(A_\gamma)\colon \nu(\pi(\Lambda))=1\}$.
We have $\nu_{\gamma}\in M'(A_\gamma)$ and $\nu_{\gamma,n}\in M'(A_\gamma)$ for all $n\in\mathbb N$.
Lemma~\ref{factor-prop} allows us to define a continuous map $\rho\colon A_{\alpha,\beta}\cap\pi(\Lambda)\to [0,1]^3$ by $\rho(\omega)\in\pi^{-1}(\omega)$.
Since $\nu\in M'(A_\gamma)\mapsto \nu\circ\rho^{-1}=\nu\circ\pi\in M([0,1]^3)$ is continuous,
$\mu_{\gamma,n}=\nu_{\gamma,n}\circ\pi$ and
$\mu_\gamma=\nu_\gamma\circ\pi$,
it follows that
$\{\mu_{\gamma,n}\}$ converges to $\mu_\gamma$ as required in Theorem~\ref{theorema}.
 The last convergence in Theorem~\ref{theorema} follows from the first two and the equalities $\#{\rm Per}_{\alpha,n}(f)=\#{\rm Per}_{\alpha,n}(\sigma)=\#{\rm Per}_{\beta,n}(\sigma )=\#{\rm Per}_{\beta,n}(f )$ for all $n\in\mathbb N$.\qed

\subsection*{Acknowledgments}
This research was supported by the JSPS KAKENHI 23K20220.

      \bibliographystyle{amsplain}

\end{document}